\documentclass{amsart}
\usepackage{amssymb,amscd,amsthm}
\usepackage{amsmath}
\usepackage{latexsym}
\usepackage[dvips]{graphicx}
\date{\today}

\newtheorem{theorem}{Theorem}
\newtheorem{remark}{Remark}
\newtheorem{lemma}{Lemma}
\newtheorem{proposition}{Proposition}
\newtheorem{corollary}{Corollary}
\newtheorem{definition}{Definition}
\begin{document}
\title[Piecewise linear Lorenz maps]
{Renormalization and conjugacy of piecewise linear Lorenz maps}

\author[Hong-Fei Cui]{Hong-Fei Cui}

\address{Wuhan Institute of Physics and Mathematics, The Chinese Academy of Sciences, Wuhan 430071, P. R. China,}

\email{cuihongfei05@mails.gucas.ac.cn}

\author[Yi-Ming Ding]{Yi-Ming Ding}

\address{Wuhan Institute of Physics and Mathematics, The Chinese Academy of Sciences, Wuhan 430071, P. R. China,}

\email{ ding@wipm.ac.cn}

\thanks{Corresponding author:Yi-Ming Ding.}
\thanks{{\it Mathematical classification (2000):} 37E05,  37F25.}
%\subjclass{37E05, 54H20.}
 %\keywords{}

\maketitle

%% Enter the first author's name and address:
%\medskip
%{\footnotesize
% \centerline{Wuhan Institute of Physics and Mathematics}
 % \centerline{The Chinese Academy of Sciences}
  % \centerline{P.O. Box 71010}
  % \centerline{Wuhan 430071, China}}
 %% Do not forget to end the {\footnotesize by the sign }

\medskip

%% The name of the associate editor will be entered by a editorial staff
 %\centerline{(Communicated by   \ \ \ \     )}
 \medskip

%% Abstract
\begin{abstract}
For each piecewise linear Lorenz map that expand on average, we show
that it admits a dichotomy: it is either periodic renormalizable or
prime. As a result, such a map is conjugate to a
$\beta$-transformation.
\end{abstract}

%% French abstract
%\begin{altabstract}
%Ceci est le r\'esum\'e fran\c cais.
%\end{altabstract}

\maketitle
\section{Introduction}\ \
%\newpage

\setcounter{equation}{0} Lorenz maps are one-dimensional maps with a
single discontinuity, which  arise  as Poincar\'e return maps for
flows on branched manifolds that model the strange attractors of
Lorenz systems. More precisely, $f:I \to I$ is a {\em Lorenz map} if
there is a point $c$ in the interior of the interval $I$ and $f$ is
continuous and increasing on both sides of $c$, and $f(\{c_-,\
c_+\}) \to \partial I$, where $f(c_+)$ and $f(c_-)$ are the one side
limits of $f$ at $c$. We are interested with piecewise linear Lorenz
maps of the form
\begin{equation}\label{LinearLorenzMap}
f_{a,b,c}(x)= \left \{ \begin{array}{ll}
ax+1-ac & x \in [0,\  c) \\
b(x-c) & x \in (c,\  1].
\end{array}
\right.
\end{equation}
The average slope of $f_{a,b,c}$ is $\int f'_{a,b,c}(x)
dx=ac+b(1-c)$. We say that $f_{a,b,c}$ {\it expand on average} if
the average slope $ac+b(1-c)$ is greater than $1$. It is easy to see
that the average slope is greater than $1$ if and only if
$f_{a,b,c}(0)<f_{a,b,c}(1)$. We are concerned with the
renormalization and conjugacy of piecewise linear Lorenz map that
expand on average. Denote by $\mathcal{L}$ as the set of piecewise
linear Lorenz maps that expand on average. Note that for $f_{a,b,c}
\in \mathcal{L}$ we may have $a<1<b$ or $a>1>b$ because we only
assume $ac+b(1-c)>1$. In both cases, $f_{a,b,c}$ is contractive on
some interval.

The map $T_{\beta, \alpha}$ defined by
$$
T_{\beta, \alpha}=\beta x + \alpha \ \ \ \mod  1
$$
is called a $\beta$-transformation (see \cite{G}). When $1 < \beta
\leq 2$, $0 \le \alpha<1$,  $T_{\beta, \alpha}=f_{\beta, \beta, c}$
with $c=(1-\alpha)/\beta$.

The study of $\beta$-transformation goes back to R$\acute{e}$nyi.
Based on bounded distortion principe, R$\acute{e}$nyi proved that
$\beta$-transformation admits an acip (absolutely continuous
invariant probability measure with respect to the Lebesgue measure).
Gelfond \cite{Gel} and Parry \cite{P1, P2} obtained the expression
of the density of the acip. Flatto and Lagarias \cite{FL1, FL2, FL3}
studied the lap counting functions. For $f \in \mathcal {L}$, we
proved in \cite{DFY} that such a map admits an ergodic acip because
there exists a positive integer $n$ so that
$(f^n_{a,b,c})'(x)>\lambda>1$ for all $x\in I$ except countable
points. Such a map is {\it expanding} in the sense that $\cup_{n \ge
0}f^{-n}(c)$ is dense in $I$.

\subsection{Renormalization of expanding Lorenz map}\ \

Renormalization is a central concept in contemporary dynamics. The
idea is to study the small-scale structure of a class of dynamical
systems by means of a renormalization operator $R$ acting on the
systems in this class. This operator is constructed as a rescaled
return map, where the specific definition depends essentially on the
class of systems. A Lorenz map $f:I \to I$ is said to be {\it
renormalizable} if there is a proper subinterval $[u,\ v] \ni c$ and
integers $\ell, r>1$ such that the map $g: [u,\ v] \to [u,\ v]$
defined by
\begin{equation} \label{renormalization}
g(x)=\left \{
\begin{array}{ll}
f^{\ell}(x) & x \in [u,\  c), \\
f^{r}(x) & x \in (c,\  v],
\end{array}
\right.
\end{equation}
is itself a Lorenz map on $[u,\ v]$. The interval $[u, \ v]$ is
called the {\it renormalization interval}. If $f$ is not
renormalizable, it is said to be {\it prime}.

A renormalization $g=(f^{\ell}, f^{r})$ of $f$ is said to be {\it
minimal} if for any other renormalization $(f^{\ell'}, f^{r'})$ of
$f$ we have $ \ell' \ge \ell$ and $r' \ge r$ (e.g. \cite{GS, MM}).
It is not an easy problem to determine wether $f$ is renormalizable
or not. In fact, it is impossible to check if $f$ is prime or not in
finite steps, because $\ell$ and $r$ in (\ref{renormalization}) may
be large.

The renormalization theory of expanding Lorenz maps is well
understood (see for example, in \cite{D, GS, MM}). We recall some
results from \cite{D} for completeness. Let $f$ be an expanding
Lorenz map. A subset $E$ of $I$ is completely invariant under $f$ if
$f(E)=f^{-1}(E)=E$, and it is proper if $E \neq I$. According to
Theorem A in \cite{D}, there is a one-to-one correspondence between
the renormalizations and proper completely invariant closed sets of
$f$. In fact, let $E$ be a proper completely invariant closed set of
$f$, put
\begin{equation}
\label{periodic pt} e_{-}=\sup\{x\in E: x<c\},\ \ \ \ \
e_{+}=\inf\{x\in E: x>c\},
\end{equation}
$\ell$ and $r$ be the maximal integers so that $f^{\ell}$ and
 $f^r$ is continuous on $(e_-,\ c)$ and $(c,\ e_+)$, respectively.
Then we have
\begin{equation} \label{periodic pts}
f^{\ell}(e_-)=e_-,\ \ \ \ f^r(e_+)=e_+,
\end{equation}
 and the map
\begin{equation} \label{completely invariant renormalization}
R_Ef(x)=\left \{ \begin{array}{ll}
f^{\ell}(x) & x \in [f^{r}(c_+),\   c) \\
f^r(x) & x \in (c, \ f^{\ell}(c_-) ]
\end{array}
\right.
\end{equation}
is a renormalization  of $f$.

So a possible way to describe the renormalizability of $f$ is to
look for the {\it minimal completely invariant closed set} of $f$.
The minimal completely invariant closed set relates to the periodic
orbit with minimal period of $f$. Suppose the minimal period of the
periodic points of $f$ is $\kappa$. It is easy to see that $f$ is
prime if $\kappa=1$ or $\kappa=\infty$. If $1<\kappa<\infty$, then
$f$ admits unique $\kappa$-periodic orbit $O$. Put
$D=\overline{\bigcup_{n\ge 0}f^{-n}(O)}$. Then we have the following
statements (see Theorem B in \cite{D}):
\begin{enumerate}
\item $D$ is the minimal completely invariant closed set of $f$.
\item $f$ is renormalizable if and only if $D \neq I$. If $f$ is
renormalizable, then $R_D$, the renormalization associated to $D$,
is the minimal renormalization of $f$.
\item We have the following trichotomy:
 i) $D=I$, ii) $D=O$, iii) $D$ is a Cantor set.
\end{enumerate}

So the minimal renormalizaion of renormalizable expanding Lorenz map
always exists. We can define a renormalization operator $R$ from the
set of renormalizable expanding Lorenz maps to the set of expanding
Lorenz maps (\cite{D, GS}). For each renormalizable expanding Lorenz
map, we define $Rf$ to be the minimal renormalization map of $f$.
For $n>1$, $R^nf=R(R^{n-1}f)$ if $R^{n-1}f$ is renormalizable. And
$f$ is $m$ ($0 \le m \le \infty$) {\it times renormalizable} if the
renormalization process can proceed $m$ times exactly.  For $0 <i
\le m$, $R^if$  is the $i$th renormalization of $f$.

\begin{definition}Let $f$ be an expanding Lorenz map.
The minimal renormalization is said to be {\it periodic} if the
minimal completely invariant closed set $D=O$, where $O$ is the
periodic orbit with minimal period of $f$. And the $i$th
renormalization $R^if$ is {\it periodic} if it is a periodic
renormalization of $R^{i-1}f$.
\end{definition}

The periodic renormalization is interesting because
$\beta$-transformation can only be renormalized periodically  (see
\cite{G}). This kind of renormalization was studied by
Alsed$\grave{a}$ and Falc$\grave{o}$ \cite{AF}, Malkin \cite{M}. It
was called phase locking renormalization in \cite{AF} because it
appears naturally in Lorenz map whose rotational interval
degenerates to a rational point.

Let $f$ be an expanding Lorenz map with a discontinuity $c$, $P_L$
be the largest $\kappa-$periodic point less than $c$ and $P_R$ be
the smallest $\kappa-$periodic point greater than $c$.  Then we have
 the following statements (\cite{D}):
\begin{enumerate}
\item The minimal renormalization of $f$ is periodic if and only if
\begin{equation} \label{periodic renormalization}
[f^{\kappa}(c_+),\ f^{\kappa}(c_-)] \subseteq [f^{\kappa}(P_L),\
f^{\kappa}(P_R)].
\end{equation}

\item One can check if the minimal renormalization of $f$ is periodic or not in following
steps:
\begin{itemize}
\item Find the minimal period $\kappa$ of $f$ by considering the preimages of
$c$, see Lemma \ref{minimal period};

\item Find the $\kappa$-periodic orbit;

\item Check if the inclusion (\ref{periodic renormalization}) holds
or not.

\end{itemize}
\end{enumerate}

So the periodic renormalization in Lorenz map plays a similar role
as the period-doubling renormalization in unimodal map.

\subsection{Main result and ideas of proof}

The main purpose of this note is to characterize the
renormalizations of $f \in \mathcal {L}$.

\vspace{0.5cm} {\bf Main Theorem.} {\it Let $f \in \mathcal {L}$,
then each renormalization of $f$ is periodic. Furthermore, $f$ is
conjugate to a $\beta$-transformation.} \vspace{0.5cm}

Follows from Milnor and Thurston \cite{MT}, a Lorenz map $f$ is
semi-conjugate to a $\beta$-transformation. According to Parry
\cite{P3}, $f$ is conjugate to a $\beta$-transformation if $f$ is
strongly transitive. Since an expanding Lorenz map is strongly
transitive if and only if it is prime \cite{D}, it is interesting to
know when a renormalizable expanding Lorenz map is conjugate to a
$\beta$-transformation.

Periodic renormalization is relevant to the conjugacy problem.
Glendinning \cite{G} showed that an expanding Lorenz map is
conjugate to a $\beta$-transformation if its renormalizations admit
some special forms. In our words, he obtained the following
Proposition.

\begin{proposition} \label{Glendinning}
 (\cite{G}) An expanding Lorenz map $f$ is conjugate to a
$\beta$-transformation if and only if $f$ is finitely renormalizable
and each renormalization of $f$ is periodic.
\end{proposition}

In fact, we shall actually prove the following Main Theorem'.

\vspace{0.3cm}

 {\bf Main Theorem'.} {\it Let $f\in \mathcal{L}$, then  $f$ is finitely renormalizable
and each renormalization of $f$ is periodic.} \vspace{0.5cm}

\begin{remark}

\begin{enumerate}

\item Main Theorem' indicates that the renormalization process of $f\in \mathcal{L}$ is simple:\ all of the
renormalizations are periodic. And one can obtain all of the
renormalizations in finite steps.

\item Suppose $f \in \mathcal
{L}$ is $m$-renormalizable, then by Theorem C in \cite{D}, $f$
admits a cluster of completely invariant closed sets
$$\emptyset=E_0 \subset E_1 \subset E_2 \subset \cdot \cdot \cdot \subset E_m \subset I,$$
where $m$ is finite, and $E_{m-i}$ equals to the  $ith$ derived set
of $E_m$, $i=1, 2, \ldots, m$.

\item According to Parry \cite{P4}, when $a \in
(2^{2^{-(m+1)}},\ 2^{2^{-m}}]$, the symmetric piecewise linear
Lorenz map $f_{a,a,1/2}$ is $m$-renormalizable, so one can obtain
countable set with given finite depth in dynamical way.

\item $f \in \mathcal {L}$, $E$ be a proper complete invariant closed set of $f$, and $g=R_E$ be the renormalization
corresponds to $E$. Since $E$ is countable, the topological entropy
$h(f|_E)=0$ (cf. \cite{D,GH,LM}). Such a renormalization does not
induce phase transition under the natural potential $-t \log |Df|$
(\cite{Do}).

\end{enumerate}
\end{remark}

Let us point out the main ideas in the proof of our Main Theorem'.
Denote by $\mathcal {LR}$ the class of maps in $\mathcal {L}$ which
are renormalizable, and $\mathcal {L}_2$ be the class of maps in
$\mathcal {L}$ and satisfy the additional condition
\begin{equation}\label{k=2}
(AC)\ \ \ \ \ \ \ \ \ \ \ \ \ \ \ \ \ 1-ac=f(0)<c<f(1)=b(1-c).\ \ \
\ \ \ \ \ \ \ \ \ \ \ \ \ \ \ \
\end{equation}
According to Lemma \ref{minimal period} in Section 2, any map in
$\mathcal {L}_2$ admits minimal period $\kappa=2$.  Fix $f \in
\mathcal {L}$, we denote $\kappa$ as its the minimal period, $O$ as
the unique $\kappa$-periodic orbit and $D$ as the minimal completely
invariant closed set of $f$.

Observe that $f \in \mathcal {LR}$ implies the minimal
renormalization $Rf \in \mathcal {L}$. So, in order to show each
renormalization of $f $ is periodic, it is necessary to show the
following
\begin{equation}\label{pperiodic}
  \forall f \in \mathcal {LR}, \ \ \ \ \ Rf \ is \ periodic.
\end{equation}
According to the trichotomy of expanding Lorenz maps,
(\ref{pperiodic}) is implied by the following dichotomy
\begin{equation}\label{alternative}
{\bf Dichotomy:} \ \ \ \ \ \ \ \ \ \  If \  f \in \mathcal {L},\
then  \ \ either \ \ D=O \ \  or \ \  D=I. \ \ \ \ \ \ \ \ \ \ \ \
\end{equation}

So, our aim is to show the Dichotomy, because, as we shall see, $f$
is finitely renormalizable is a direct consequence of it. This,
together with Proposition \ref{Glendinning}, ensures the conjugacy.

The first step towards the proof of the Dichotomy is to reduce the
proof for maps in $\mathcal {L}$ to the maps in $\mathcal {L}_2$ by
trivial renormalization (see Section 2 for the details of trivial
renormalization). In what follows, we sketch the proof of Dichotomy
for $f \in \mathcal {L}_2$.

According to equations (\ref{periodic pt}) and (\ref{periodic pts}),
any renormalization corresponds two periodic points, $e_-$ and
$e_+$. An $m$-periodic point is said to be {\it nice} if $f^m$ is
continuous on the interval between $p$ and the critical point $c$.
$\{p,\ q\}$ is a {\it nice pair} if both $p$ and $q$ are nice
periodic points and $p<c<q$.
%Since $f$ is piecewise linear, it is useful to consider the multiples of nice
%periodic points.
Let $\{p,\ q\}$ be a nice pair, and the period of $p$ and $q$ be
$\ell$ and $r$, respectively. Put
$$M_p=\prod_{i=0}^{\ell-1}f'(f^i(p)),\ \ \ \ \ M_q=\prod_{i=0}^{r-1}f'(f^i(q)).$$
Each factor in $M_p$ and $M_q$ is either $a$ or $b$ because $f$ is
piecewise linear. The proof of the Dichotomy for  $f \in \mathcal
{L}_2$ can be divided into two steps:

 {\it Step 1:} Show that if the nice pair $\{p,\ q\}$ corresponds to a
renormalization, then
$$(M_{p}-1)(M_q-1) \le 1. $$

{\it Step 2:} If $D \neq O$, show that for any nice pair $\{p,\
q\}$, we have
\begin{equation} \label{lowerbd} (M_p-1)(M_q-1)>1.
\end{equation}

Step 1 is fairly easy, and depends on the properties of
renormalization and $f$ is piecewise linear.

Step 2 is more involved. We decompose the proof into three cases:
both $a \ge 1$ and $ b \ge 1$, $a<1<b$ and $a>1>b$. In the first
case, all of the factors in the product of $M_p$ and $M_q$ are no
less than $1$, it is easier to get the lower bounds of $M_p$ and
$M_q$. The first case is a direct consequence of some inequalities
obtained from the action of $f$ on some intervals. The second case
and the third case are similar. In order to get lower bounds for
$M_p$ and $M_q$ when $a<1<b$, we introduce the {\it first exit
decomposition}. Although $f$ is contractive on the left side of the
critical point, it is possible to find a set $A$ ($A=[0,\ c_1]$,
$c_1$ is the preimage of $c$ on the left side of $c$) so that
$M_A(x) \ge 1$ for many initial $x$, where
$$
M_A(x)=\prod_{i=0}^{n_A(x)-1}f'(f^i(x)),
$$ and $n_A(x)$ is the
first exit time of the orbit  $O(x)$ from $A$.

Suppose the orbit $O(c_-)$ leave $A$ exact $s$ times, and the orbit
$O(c_+)$ leaves $A$ exact  $t$ times, using the first exit
decomposition, we can obtain (see Section 3 for details)
$$M_p=M_A(c_-)M_A(y_1)\cdots M_A(y_{t-1})W(y_t),
$$
$$
M_q=M_A(c_+)M_A(x_1)\cdots M_A(x_{s-1})W(x_s).$$ Depending on the
position of $f(0)=1-ac$, we have three cases. In each case, we can
obtain lower bounds of $M_p$ and $M_q$ to ensure (\ref{lowerbd}).

The remain parts of the paper is organized as follows. We describe
trivial renormalization in Section 2, so that we can reduce the
proof for maps in $\mathcal {L}$ to the maps in $\mathcal {L}_2$. We
set up the {\it expansion of nice pair} (\ref{lowerbd}) for maps in
$\mathcal {L}_2$ in Section 3, and prove  Main Theorem' in the last
section.

%\vspace{0.3cm}

\vspace{0.5cm}

\section{Trivial renormalization} \ \

In the definition of renormalization of Lorenz map, we assume that
both $\ell>1$ and $r>1$. And we have a one-to-one correspondence
between such kind of renormalizations and  proper completely
invariant closed sets (Theorem A in \cite{D}).

\begin{definition}(\cite{GS})
A Lorenz map $f$ is said to be {\it trivially renormalizable} if we
have $(\ell,\ r)=(1,\ 2) $ or $(\ell,\ r)=(2,\ 1)$ in equation
(\ref{renormalization}), and such a map $g$ is called a trivial
renormalization of $f$.
\end{definition}

%The trivial renormalization can be characterized by the minimal
%period of $f$.

\begin{lemma} (\cite{D}) \label{minimal period}
Suppose $f$ is an expanding Lorenz map on $[a,\ b]$ without fixed
point. Then the minimal period of $f$ is equal to $\kappa=m+2$,
where
\begin{equation} \label{mm}
 m=\min\{i \ge 0: f^{-i}(c) \in [f(a), f(b)]\}.
\end{equation}
\end{lemma}

\begin{proposition}\label{trivial renormalization}  Let $f$ be
an expanding Lorenz map on $[a,\ b]$ with minimal period $\kappa$.
If $c \notin (f(a),\ f(b))$, then there exists a Lorenz map $g$ with
minimal period less than $\kappa$, such that $f$ is renormalizable
if and only if $g$ is renormalizable. Moreover, if $f$ is
renormalizable, then the minimal renormalization of $f$ is periodic
if and only if the minimal renormalization of $g$ is periodic.
\end{proposition}

\begin{proof}
Since $c \notin (f(a),\ f(b))$, we have two cases: $c \le f(a)$ or
$c \ge f(b)$.

For the case $c \le f(a)$,  the following map
\begin{eqnarray*} g(x)=\left \{ \begin{array}{ll}
f^{2}(x) & x \in [a,\   c) \\
f(x) & x \in (c, \ f(b) ].
\end{array}
\right.
\end{eqnarray*}
is an expanding Lorenz map with minimal period less than $\kappa$,
and
\begin{equation}\label{trivial0}
orb(x, g)=orb(x, f)\cap [a,\ f(b)].
\end{equation}

If $c \ge f(b)$, the following
\begin{eqnarray*} g(x)=\left \{ \begin{array}{ll}
f(x) & x \in [f(a),\   c) \\
f^2(x) & x \in (c, \ b ].
\end{array}
\right.
\end{eqnarray*}
is also an expanding Lorenz map with minimal period less than
$\kappa$, and
\begin{equation}  \label{trivial}
orb(x, g)=orb(x, f)\cap [f(a),\ b].\end{equation}

See Figure 2 (Heavy Lines) for the intuitive pictures of $g$.

\begin{figure}
\centering
\includegraphics[width=\textwidth]{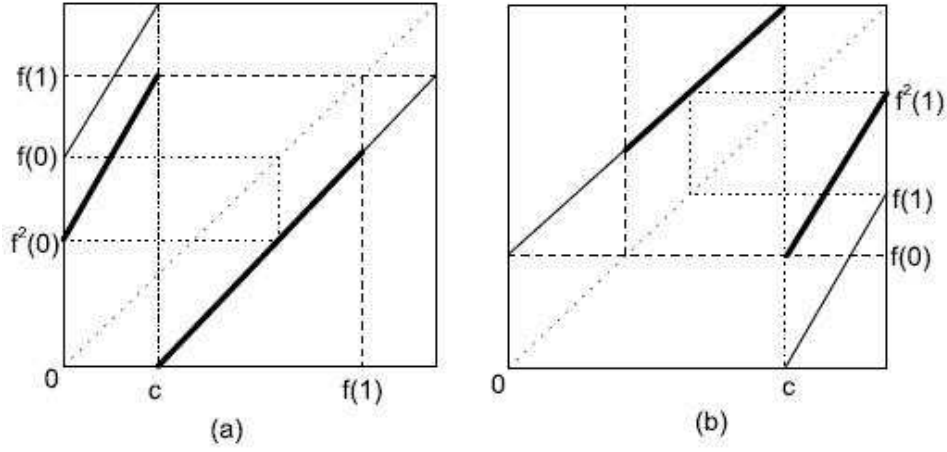}
\caption{Trivial renormalization of a map on $[0,\ 1]$, the pictures
of $g$: (a) $c \le f(0)$,\ \ (b) $f(1)\le c$.}
\end{figure}

Denote $O_f$ and  $O_g$ as the periodic orbit with minimal period of
$f$ and $g$,  and $D(f)$ and $D(g)$ as the minimal completely
invariant closed set of $f$ and $g$, respectively.

If $c \le f(a)$, by (\ref{trivial0}), we get  $O_g=O_f \cap [a,\
f(b)]$, and $D(g)=D(f) \cap [a,\ f(b)]$. It follows that $D(f)=I$ is
if and only if $D(g)=[a,\ f(b)]$, and $D(f)=O_f$ if any only if
$D(g)=O_g$.

If $c \ge f(b)$, by (\ref{trivial}), we obtain $O_g=O_f \cap [f(a),\
b]$, and $D(g)=D(f) \cap [f(a),\ b]$. It follows that $D(f)=I$ is if
and only if $D(g)=[f(a),\ b]$, and $D(f)=O_f$ if any only if
$D(g)=O_g$.

In both cases, according to Theorem B in \cite{D}, we know that $f$
is renormalizable if and only if $g$ is renormalizable. Moreover, if
$f$ is renormalizable, the minimal renormalization of $f$ is
periodic if and only if the minimal renormalization of $g$ is
periodic.
\end{proof}

%\begin{remark} \begin{enumerate}
%\item

It is easy to see that a Lorenz map with $c \in (f(a),\ f(b))$ can
not be trivially renormalizable, so the statement in Proposition
\ref{trivial renormalization} is just the the fact that an expanding
Lorenz map $f$ is {\it trivially renormalizable} if and only if $c
\notin (f(a),\ f(b))$.

%\item

%\end{enumerate}
%\end{remark}

Applying trivial renormalization (see Proposition \ref{trivial
renormalization}, (\ref{trivial0}) and (\ref{trivial}))
consecutively if possible, we get the following Corollary.

\begin{corollary} \label{cor}  Let $f$ be
an expanding Lorenz map with minimal period $\kappa$. If
$\kappa<\infty$, then $f$ can be trivially renormalized finite times
to be an expanding Lorenz map $g$ with $\kappa(g) \le 2$.

%Moreover, $\kappa(g)=2$ if and only if $\{ c_+,\ c_- \} \bigcap
%O=\emptyset$, where $O$ is the $\kappa$-periodic orbit of $f$.
\end{corollary}
%\vspace{0.2cm}

%Since $\kappa=2$ if and only if $c \in [f(a),\ f(b)]$ (cf. Lemma
%\ref{minimal period}), one can obtain the minimal period of $f$ via
%consecutive trivial renormalizations.

\vspace{0.5cm}

\section{Expansion of nice pair} \ \

Suppose $p$ is a periodic point with period $m$. $p$ is called a
{\it nice periodic point} if $f^m$ is continuous on the interval
between $p$ and the critical point $c$. $\{p,\ q\}$ is called a {\it
nice pair} if $p<c<q$, and both $p$ and $q$ are nice periodic
points. If $E$ is a proper completely invariant closed set of $f$,
$e_-$ and $e_+$ are defined by (\ref{periodic pt}), then $\{e_-,\
e_+\}$ is a nice pair. A nice pair $\{p,\ q\}$ corresponds to a
renormalization if and only if $[f^r(c_+),\ f^{\ell}(c_-)] \subseteq
[p,\ q]$, where $\ell$ and $r$ are the periods of $p$ and $q$,
respectively.

Assume that $f \in \mathcal {L}_2$, by Lemma \ref{minimal period},
$f$ admits a two periodic orbit $O=\{P_L,\ P_R\}$, and
$0<P_L<c<P_R<1$. Let $\{p,\ q\}$ be a nice pair of $f$, $\ell$ and
$r$ be the period of $p$ and $q$, respectively. So $f^{\ell}$ is
linear on $[p,\ c_-]$, and $f^r$ is linear on $[c_+,\ q]$. Put
\begin{equation}  \label{mp-mq}
\left\{
\begin{aligned}
         M_{p}:&=(f^{\ell})'(p)=(f^{\ell})'(c_-)=\prod_{i=0}^{\ell-1}f'(f^i(c_-)), \\
                  M_{q}:&=(f^{r})'(q)=(f^{\ell})'(c_+)=\prod_{i=0}^{r-1}f'(f^i(c_+)).
                          \end{aligned} \right.
                          \end{equation}

The main purpose of this section is to prove the following {\it
expansion of nice pair} for maps in $\mathcal {L}_2$, which is
essential for us to obtain the Dichotomy (\ref{alternative}).

\begin{theorem} \label{lower bound}
Suppose $f \in \mathcal {L}_2$, $\{p,\ q\}$ is a nice pair of $f$,
and $M_{p}$ and $M_q$ are defined as above. If\ \  $[f(0),\ f(1)]
\nsubseteq [P_L,\ P_R]$, then
\begin{equation}\label{expansion}
(M_{p}-1)(M_q-1)>1.
\end{equation}

\end{theorem}

\begin{remark}
By (\ref{periodic renormalization}) and the trichotomy claimed by
Theorem B in \cite{D} , $[f(0),\ f(1)] \nsubseteq [P_L,\ P_R]$ is
equivalent to $D \neq O$.
\end{remark}
\vspace{0.3cm}

The proof of Theorem \ref{lower bound} is technical. Let $f \in
\mathcal {L}_2$ such that $D \neq O$, we divide the proof into three
cases: both $a \ge 1$ and $ b \ge 1$, $a<1<b$ and $a>1>b$. In the
first case, all of the factors in the product of $M_p$ and $M_q$ are
no less than $1$, it is easier to get the lower bounds of $M_p$ and
$M_q$. In fact, the expansion of a nice pair (\ref{expansion}) can
be achieved by Lemma \ref{inequality}, which is a direct consequence
of some inequalities obtained from the action of $f$ on some
intervals. The second case and the third case are similar. In order
to get a lower bound for $M_p$ and $M_q$ when $a<1<b$, we introduce
the {\it first exit decomposition}. Although $f$ is contractive on
the left side of the critical point, we try to decompose $M_p$ and
$M_q$ into parts so that each part is no less than $1$. Depending on
the position of $f(0)=1-ac$, we have three cases. In each case, we
can obtain lower bound of $M_p$ and $M_q$ to ensure
(\ref{expansion}). In the remain parts of this section, we introduce
the {\it first exit decomposition} firstly, then we prove some
technical Lemmas based on the detailed dynamics of $f$, and prove
Theorem \ref{lower bound} finally.

\subsection{First exit decomposition}

Let $A$ be a given set, $O(x)=\{f^j(x); j\ge 0\}$ be the orbit with
initial $x$. If $O(x)$ visits $A$, denote
$$
n_A(x)=\min\{k: f^{k-1}(x) \in A,\ f^k(x) \notin A\}
$$
as the first exit time of $O(x)$ from $A$, and the $s$th $(s \ge 1)$
exit time $n_s(x)$ from $A$ are defined inductively by
$$
n_1(x):=n_A(x),\ \ n_s(x):=\min\{k>n_{s-1}: f^{k-1}(x) \in A,\
f^k(x) \notin A\}.
$$
If $O(x)$ does not visit $A$, $n(x)=\infty$.

Denote $x_s:=f^{n_s}(x), \ \ s=1, 2,\ldots$. Put
$$
M_A(x)=\prod_{j=0}^{n_A(x)-1}f'(f^j(x)).
$$
%and $M_A(x)=\infty$ if $n_A(x)=\infty$.

Using above notations, the following {\it first exit decomposition}
is trivial.

\begin{lemma} \label{decomposition}$x \in I$, and $n_{s}(x) \le n<n_{s+1}(x)$,
\begin{equation}\label{deco}
(f^n)'(x)=\prod_{j=0}^{n-1}f'(f^j(x))=M_A(x)M_A(x_1)\cdots
M_A(x_{s-1}) W(x_s),
\end{equation}
where
$$
W(x_s)=f'(x_s)f'(f(x_s))\cdots f'(f^{n-1}(x)),
$$
and $W(x)=1$ if and only if $x_s=f^{n}(x)$.
\end{lemma}

\vspace{0.3cm}
\subsection{Technical Lemmas}

Suppose $f:=f_{a,b,c} \in \mathcal {L}_2$.
 Denote the $2$-periodic points are $P_L$ and $P_R$, $0<P_L<c<P_R<1$,
 and $c_*$ and $c^*$ are the preimages of $c$, $0<c_*<P_L<c<P_R<c^*<1$. By direct
 calculations, we get
 \begin{equation} \label{eq:1}
  \begin{aligned}
        P_L=\frac{b(c-(1-ac))}{ab-1},\ \ \ \ \ P_R=\frac{abc-(1-ac)}{ab-1}\\
        c_*=\frac{c-(1-ac)}{a},\ \ \ \ \ \ \ \ \ \ \ \ \ \  \ \ \ \ \ c^*=\frac{c+bc}{b}.
                           \end{aligned}
                           \end{equation}
Observe that $f^2$ is linear (with slope $ab=f^2(P_L)>1$) on $[c_*,\
P_L]$, and $f^2(P_L)=P_L$. Track the preimages of $c_*$ on $[c_*,\
P_L]$, one can get an increasing sequence $\{c_n\} \subset [c_*,
P_L]$,
\begin{equation} \label{c_n}
c_0:=c_*,\ \ f^2(c_1)=c_0,\   \cdots,\ f^2(c_n)=c_{n-1},\ \cdots
\end{equation}
and $c_n \uparrow P_L$. $(c_*, P_L)=\bigcup_{k \ge 1}(c_{k-1},\
c_k]$. Similarly, there exists a decreasing sequence $\{c_n'\}$
approaches to $P_R$ so that
\begin{equation} \label{c_n'}
c_0':=c^*,\ \ f^2(c_1')=c_0',\ \cdots,\ f^2(c_n')=c_{n-1}',\ \cdots.
\end{equation}

\begin{lemma} \label{ineq-1} Let $\{c_n\}$ and $\{c_n'\}$ are defined as (\ref{c_n}) and (\ref{c_n'}), we have
\begin{equation} \label{length-11}
|(c_{n-1},\ c_n)| \le |(c_n,\ c)|,
\end{equation}
\begin{equation} \label{length-12}
|(c'_{n},\ c_{n-1})| \le |(c,\ c'_n)|.
\end{equation}
\end{lemma}

\begin{proof}At first, we prove (\ref{length-11}). Using
(\ref{eq:1}),
$$
|(c_*,\ P_L)|=\frac{c-(1-ac)}{a(ab-1)},\  \ |(P_L,\
c)|=\frac{a(b(1-c)-c)}{c-(1-ac)}|(c_*,\ P_L)|.
$$
Since $f^{2k}$ maps $(c_k,\ P_L)$ homeomorphically to $(c_*,\ P_L)$,
$$
|(c_n,\ P_L)|=\frac{1}{a^nb^n}|(c_*,\ P_L)|, \ \ \ \ \ \ |(c_{n-1},\
P_L)|=\frac{1}{a^{n-1}b^{n-1}}|(c_*,\ P_L)|.
$$
It follows that
$$
|(c_{n-1},\ c_n)|=|(c_{n-1},\ P_L)|-|(c_n,\
P_L)|=(\frac{1}{a^{n-1}b^{n-1}}-\frac{1}{a^nb^n})|(c_*,\ P_L)|,
$$
and
$$
|(c_{n},\ c)|=|(c_n,\ P_L)|+|(P_L,\
c)|=(\frac{1}{a^nb^n}+\frac{a(b(1-c)-c)}{c-(1-ac)})|(c_*,\ P_L)|.
$$
Hence, (\ref{length-11}) is equivalent to
\begin{equation} \label{length-1}
 \frac{2}{ab}+(ab)^{n-1}\frac{a(b(1-c)-c)}{c-(1-ac)} \ge 1.
\end{equation}
Remember that $f$ satisfies the additional condition (\ref{k=2}),
i.e.,  $0<f(0)=1-ac<c<b(1-c)=f(1)<1$, $\frac{a(b(1-c)-c)}{c-(1-ac)}$
is always positive.

Since $ab>1$, it is enough to prove (\ref{length-1}) with $n=1$,
i.e.,
\begin{equation} \label{length-2}
 F(b):=\frac{2}{ab}+\frac{a(b(1-c)-c)}{c-(1-ac)} \ge 1.
\end{equation}

If $ab \le 2$, then $F(b) \ge \frac{2}{ab} \ge 1$. For the case
$ab>2$, $a$ is fixed,
$$F'(b)=-\frac{2}{ab^2}+ \frac{a(1-c)}{c-(1-ac)}=\frac{a^2b^2(1-c)-2(c-(1-ac))}{ab^2(c-(1-ac))}.$$

Using $ab>2$ and $f(1)=b(1-c)>c$,
$$
a^2b^2(1-c)-2(c-(1-ac))>a^2bc-2c+2-2ac>ac(ab-2)+2(1-c)>0.
$$
So $F'(b)>0$ when $ab>2$. It follows that $F(b)>F(\frac{2}{a})=1$.
(\ref{length-2}) holds.

For the second inequality, by similar calculations, one can see that
(\ref{length-12}) is equivalent to
\begin{equation}
\label{length-121}
 \frac{2}{ab}+(ab)^{n-1}\frac{b(c-(1-ac))}{b(1-c)-c} \ge 1.
\end{equation}
We shall prove (\ref{length-121}) with $n=1$, i.e.,
\begin{equation} \label{length-122}
G(a):=\frac{2}{ab}+\frac{b(c-(1-ac))}{b(1-c)-c} \ge 1.
\end{equation}

If $ab \le 2$, then $G(a) \ge  \frac{2}{ab} \ge 1$. When $ab>2$,
$$G'(a)=-\frac{2}{a^2b}+ \frac{bc}{b(1-c)-c}=\frac{a^2b^2c-2(b(1-c)-c)}{a^2b(b(1-c)-c)}.$$

Using $ab>2$ and $f(0)=1-ac<c$, one obtains
$$
a^2b^2c-2(b(1-c)-c)>2abc-2b(1-c)+2c>2b(c-(1-ac))+2c>0.
$$
So $G'(a)>0$ when $ab>2$. It follows that $G(a)>G(\frac{2}{b})=1$.
(\ref{length-122}) holds.

\end{proof}

\vspace{0.2cm}

\begin{lemma}\label{inequality} Let $\{c_n\}$ and $\{c_n'\}$ be defined as (\ref{c_n}) and
(\ref{c_n'}).

\begin{enumerate}
\item Suppose $f(0) \in (c_{k-1},\ c_k]$, we have
 $$ab a^{k+1}b^{k}>1+a^{k+1}b^k \ \ \ \ \ \ \ and \ \ \ \ \ \ \ a^{k+1}b^k>1.$$

\item Suppose $f(1) \in [c_{k}',\ c_{k-1}')$, we have
$$ab a^{k}b^{k+1}>1+a^{k}b^{k+1} \ \ \ \ \ \ \  and \ \ \ \ \ \ \ a^{k}b^{k+1}>1.$$
\end{enumerate}
\end{lemma}

\begin{proof}

It is necessary to prove (1), (2) can be proved similarly.

Consider the interval $(c_k,\ \ P_L)$, since $f(0) \in (c_{k-1},\
c_k]$, we have
\[
\begin{CD}
(c_k, P_L) @>{f^{2k}}>{(ab)^{k}}> (c_*, P_L) @>{f}>{a}> (c, P_R)@>{
f^2}>{ab}> (f(0), P_R) \supset (c_k, P_L) \cup (c, P_R).
\end{CD}
\]
So we have
$$f^{2k+3}((c_k, P_L))\supset (c_k, P_L) \cup (c, P_R)=f^{2k+1}((c_k, P_L))\cup (c, P_R).$$
It follows that $$a^{k+2}b^{k+1}|(c_k,\ \ P_L)|>|(c_k,\ P_L)|+|(c,\
P_R)|.$$ Notice that $|(c,\ P_R)|=a^{k+1}b^k |(c_k,\ P_L)|$, we
obtain $a^{k+2}b^{k+1}>1+a^{k+1}b^k$.

Consider the interval $(c_{k-1},\ \ c_k)$, we obtain
\[
\begin{CD}
(c_{k-1},\ \ c_k) @>{f^{2(k-1)}}>{(ab)^{k-1}}> (c_*,\ c_1)
@>{f^2}>{ab}> (0,\ c_*)@>{ f}>{a}> (f(0),\ c).
\end{CD}
\]

Similarly, it follows $$a^{k+1}b^{k}|(c_{k-1},\ \ c_k)|=|(f(0),\
c)|.$$ By Lemma \ref{ineq-1} and the condition that $f(0) \in
(c_{k-1}, \ c_{k}]$,
$$
a^{k+1}b^{k}=\frac{|(f(0),\ \ c)|}{|(c_{k-1},\ \
c_k)|}>\frac{|(c_k,\ c)|}{|(c_{k-1},\ c_k)|} \ge 1.
$$
\end{proof}

\begin{lemma}\label{M(x)} Suppose $a<1<b$, $A=[0,\ c_*]$,
$$M(x):=M_A(x)=\prod_{i=0}^{n_A(x)-1}f'(f^i(x)),$$  where
$n_A(x)$ is the first exit time of the orbit  $O(x)$ from $A$. If
$f(0) \in (c_*,\ c)$, then
\begin{equation}\label{M(x)-0}
M(x):=M_A(x)>1,\ \ \ \ \ \forall \ x  \ge f(0).
\end{equation}

Similarly, suppose $a>1>b$, $B=[c^*,\ 1]$,
$M_B(x)=\prod_{i=0}^{n_B(x)-1}f'(f^i(x))$, where $n_B(x)$ is the
first exit time of the orbit  $O(x)$ from $B$. If $f(1) \in (c,\
c^*)$, then
\begin{equation}\label{M(x)-00}
M_B(x)>1,\ \ \ \ \ \forall\ x\le f(1).
\end{equation}
\end{lemma}

\begin{proof} We only prove the Lemma for case $a<1<b$, the proof can
adapt to the case $a>1>b$ easily.

Since $f(x)> c$ for all $x \in (c_*,\ c)$ and $ab>1$, we know that
$M(x)=\infty$ when $n_A(x)=\infty$. In what follows, we show that
$M(x) > 1$ for $x \in I$ with $n_A(x)<\infty$.

The main reason for us to consider the first exit decomposition with
respect to $A=[0,\ c_*]$ is that $f$ maps $(c_*,\ c)$
homeomorphically to $(c, \ 1)$, which implies that any orbit with
initial position $x \notin A$ can not stay on the left of $c$ two
consecutive times before it visits $A$. This fact is useful for us
to obtain lower bound of $M_A(x)$.

When $f(0)>c_*$, each orbit of $f$ can stay on the left of $c$ at
most two consecutive times. To check (\ref{M(x)-0}), we consider
three cases:

If $x \ge c_+$, the product $M(x)$ begin with $b$ and end with only
one $a$, and it can not have two consecutive $a$. So $M(x)>1$
because $ab>1$.

If $x \in (P_L,\ c_-]$, then $f(x) \in (P_R,\ 1]$. There is a
nonnegative integer $m$ such that $f^{2m}(f(x)) \ge c^*$. So
$f^{2m+2}(x) \ge c$ and $M(f^{2m+2}(x))> 1$. It follows
$$M(x)=(ab)^{m+1}M(f^{2m+2}(x))> 1.$$

If $f(0) \in (c_*,\ P_L)$, there exists positive integer $k$ so that
$f(0)\in (c_{k-1},\ P_L)$. For $x \in (f(0),\ P_L)$,  one can see
$M(x)=(ab)^ma^{k+1}b^k$ for some $m \ge 0$. By Lemma
\ref{inequality},
$$M(x) \ge a^{k+1}b^k > 1.$$
\end{proof}

Let $ i=\min\{k: f^k(0)>c\}$ be the least integer so that
$f^i(0)>c$. Each orbit of $f$ can stay consecutively on the left of
$c$ at most $i$ times. $f(0) \le c_*$ implies $i \ge 3$.

Let $ j=\min\{k: f^k(1)<c\}$ be the least integer so that
$f^i(1)<c$. $f(1) \ge c^*$ implies $j \ge 3$.

\begin{lemma}\label{baj} Let $i$ and $j$ be defined as above, we have
\begin{equation}\label{baj1}
ba^{i-1}>1+a+\cdots+a^{i-2},
\end{equation}
\begin{equation}\label{baj2}
ab^{j-1}>1+b+\cdots+b^{j-2}.
\end{equation}
\end{lemma}

\begin{proof}Since $i$ is the least positive integer such that
$f^{i-1}(0)<c<f^i(0)$, by direct calculation,
$$
f(0)=1-ac, f^2(0)=(1-ac)(1+a),\  \ldots,
f^{i-1}(0)=(1-ac)(1+a+\cdots+a^{i-2})<c.
$$
It follows
$$
c>\frac{1+a+\cdots+a^{i-2}}{1+a+\cdots+a^{i-1}}.
$$
On the other hand, by assumption (\ref{k=2}), $c<f(1)=b(1-c)$
implies $c<\frac{b}{1+b}$. We get
$$
\frac{1+a+\cdots+a^{i-2}}{1+a+\cdots+a^{i-1}}<\frac{b}{1+b},
$$
which is equivalent to (\ref{baj1}).

(\ref{baj2}) can be proved by similar calculations.
\end{proof}

Remember that $c_1$ and $c_1'$ are defined by (\ref{c_n}) and
(\ref{c_n'}).

\begin{lemma} \label{control} Let $i$ and $j$ be defined as above, we have
\begin{equation} \label{control-1}
ba^i<1\ \ \ \ \ \  implies \ \ \ \ \ \ f^{i-1}(0) \in (c_1,\ c),
\end{equation}
\begin{equation} \label{control-2}
ab^j<1 \ \ \ \ \ \  implies \ \ \ \ \ \  f^{j-1}(1) \in (c,\ c_1').
\end{equation}

\end{lemma}

\begin{proof}
We only show (\ref{control-1}). By the definition of $i$,
$$
0<f(0)<f^2(0)<\cdots<f^{i-1}(0)<c<f^i(0).
$$
Since $f^{i-1}$ maps $(0,\ f(0))$  to $(f^{i-1}(0),\ f^i(0)) \ni c$
homeomorphically, there exists $y \in (0,\ f(0))$ so that
$f^{i-1}(y)=c$.

Observe that
\[
\begin{CD}
(c_*,\ \ c_1) @>{f^{2}}>{ab}> (0,\ c_*),
\end{CD}
\]
there exists $z \in (c_*,\ \ c_1)$ such that $f^2(z)=y$.

Consider the interval $(c_*,\ \ z)$, we have
\[
\begin{CD}
(c_*,\ \ z) @>{f^{2}}>{ab}> (0,\ y) @>{f^{i-1}}>{a^{i-1}}>
(f^{i-1}(0),\ c).
\end{CD}
\]
It follows that $$ba^i|(c_*,\ \ z)|=|(f^{i-1}(0),\ \ c)|.$$

If $f^{i-1}(0)<c_1$,  by Lemma \ref{ineq-1},
$$
ba^i=\frac{|(f^{i-1}(0),\ \ c)|}{|(c_*,\ \ z)|}>\frac{|(c_1,\
c)|}{|(c_*,\ c_1)|} \ge 1.
$$
We obtain a contradiction. Hence, (\ref{control-1}) is true.
\end{proof}

\vspace{0.2cm}

\begin{lemma}\label{M(x)-1}  Suppose $a<1<b$, $A=[0,\ c_*]$, $M(x):=M_A(x)$ is defined as in (\ref{M(x)-0}). If
$f(0) <c_*$, then
\begin{equation}
M(x):=M_A(x)>1,\ \ \ \ \ \forall x\ge c_1.
\end{equation}

Similarly, Suppose $a>1>b$, $B=[c^*,\ 1]$, $M_B(x)$ is defined as in
(\ref{M(x)-00}). If $f(1) >c^*$, then
\begin{equation}
M_B(x)>1,\ \ \ \ \ \forall x\le c_1'.
\end{equation}
\end{lemma}
%\vspace{0.2cm}

\begin{proof}
Let $i$ be defined as above. If $x\in (c_1,\ c_2]$, then
$M(x)=ababa^m$ for some $0<m \le i-1$. Since $a<1<b$, we have $M(x)
\ge ababa^{i-1} \ge (ba^2)(ba^{i-1}) \ge 1$. In fact, Lemma
\ref{control}, together with $i \ge 3$, implies that both $ba^{i-1}$
and $ba^2$ are no less than $1$. The remain cases can be shown by
similar arguments in the proof of Lemma \ref{M(x)}.

\end{proof}

\subsection{Proof of Theorem \ref{lower bound}}

Now we present the proof of Theorem \ref{lower bound}.

Let $f\in \mathcal {L}_2$,  $p$ is an $\ell$-periodic point and $q$
is a $r$-periodic point of $f$, $\{p,\ q\}$ is a nice pair of $f$,
and
$$M_p=(f^{\ell})'(c_-)=\prod_{i=0}^{\ell-1}f'(f^i(c_-)),\ \ \ \ \ M_q=(f^{r})'(c_+)=\prod_{i=0}^{r-1}f'(f^i(c_+)).$$
Each factor in $M_p$ and $M_q$ is either $a$ or $b$ because $f$ is
piecewise linear.

Our aim is to show that
$$
(M_p-1)(M_q-1)>1
$$
for each nice pair $\{p,\ q\}$ provided  $[f(0),\ f(1)] \nsubseteq
[P_L,\ P_R]$. Remember that $P_L$, $P_R$, $c_*$ and $c^*$ are all
calculated in (\ref{eq:1}).

The proof can be divided into three cases: both $a \ge 1$ and $b \ge
1$, $a<1<b$, and $a>1>b$.

 \vspace{0.2cm}
 {\bf Case A:} $a\ge 1$ and $b \ge 1$.

 \vspace{0.2cm}

Since $[f(0),\ f(1)] \nsubseteq [P_L,\ P_R]$, we have $f(0)<P_L$ or
$f(1)>P_R$. Without loss of generality, we assume $f(1)>P_R$. It
follows either $f(1) \in (P_R, c^*)$ or $f(1) \ge c^*$.

If $f(1) \in (P_R, c^*)$,  there exists $k$ so that $f(1) \in
[c_{k}',\ c_{k-1}')$, by Lemma \ref{inequality}, we have
$aba^{k}b^{k+1}>1+a^kb^{k+1}$.  $p$ is a nice $\ell$-periodic point
indicates $\ell \ge 2k+3$. In fact, in this case, when $m<2k+2$, the
interval $(f^m(p),\ f^m(c_-))$ does not contain $c_*$ and $c^*$, so
$N((f^{2k+2}(p),\ f^{2k+2}(c_-))) \ge 1$. Since $a \ge 1$ and $b\ge
1$, $M_p=aba^{k}b^{k+1}\prod_{i=2k+3}^{\ell-1}f'(f^i(c_-))\ge
aba^{k}b^{k+1}$ and $M_q=ab \prod_{i=2}^{r-1}f'(f^i(c_+))\ge ab$.
Hence,
$$
(M_p-1)(M_q-1) \ge (aba^{k}b^{k+1}-1)(ab-1)>a^{k}b^{k+1}(ab-1)>1.
$$

If $f(1)\ge c^*$, by  similar arguments as above, we get $M_p \ge
ab^2>1+b$ by Lemma \ref{baj}. Hence, using $M_q \ge ab$, we obtain
$$
(M_p-1)(M_q-1) \ge (ab^2-1)(ab-1)>b(ab-1)>1.
$$

Therefore, the expansion of nice pair (\ref{expansion}) is proved
when both $a$ and $b$ are no less than $1$.

 \vspace{0.2cm}
{\bf Case B:} $a<1<b$.
 \vspace{0.2cm}

In this case, $f$ is contractive on the left side of $c$. We
consider the first exit decomposition of $M_p$ and $M_q$ with
respect to $A=[0,\ c_*]$. Since $f$ maps $(c_*,\ c)$
homeomorphically to $(c, \ 1)$, any orbit with initial position $x
\notin A$ can not stay on the left of $c$ two consecutive times
before it visits $A$.

Suppose $O_r(c_+):=\{c_+,\ f(c_+),\ \ldots, f^{r-1}(c_+)\}$ exits
$A=[0,\ c_*]$ exact $s$ ($s \ge 1$) times. Put $x_j:=f^{n_j}(c_+)$,
where $n_j$ is the $j$th exit time for the finite orbit $O_r(c_+)$
with respect to $A$. According to the first exit decomposition
(\ref{deco}),
 $$M_q=(f^r)'(c_+)=\prod_{k=0}^{r-1}f'(f^k(c_+))=M(c_+)M(x_1)\cdots
M(x_{s-1})W(x_s),$$ where $W(x_s)=f'(x_s)f'(f(x_s))\cdots
f'(f^{r-1}(c_+))$. $W(x_s) \ge 1$ because it can not contain two
consecutive $a$, the last factor is $b$, and $ab>1$.

Similarly, suppose $O_{\ell}(c_-)$ exits $A$ exact $t$ times. Denote
$y_j:=f^{n_j}(c_-)$, one gets
 $$M_p=(f^{\ell})'(c_-)=\prod_{k=0}^{\ell-1}f'(f^k(c_-))=M(c_-)M(y_1)\cdots
M(y_{t-1})W(y_t),$$ and $W(y_t)=f'(y_t)f'(f(y_t))\cdots
f^{\ell-1}(c_-) \ge 1$.

Depending on the position of $f(0)$, we distinguish three subcases:
$P_L \le f(0)<c$, $c_*<f(0)<P_L$ and $f(0)\le c_*$. We shall show
that the expansion of nice pair (\ref{expansion}) holds in each
subcase.

\vspace{0.2cm} (i) {\bf Subcase $P_L<f(0)<c$.}\vspace{0.2cm}

Since $f(x)>f(0)$ for $x \in A$, by Lemma \ref{M(x)} we know that
$M(x_j)\ge 1$, $j=1,\cdots, s-1$,  and $M(y_j) \ge 1$, $j=1,\cdots,
t-1$. It follows that $M_q\ge M(c_+)=ab$ and $M_p \ge M(c_-)$. Since
$[f(0),\ f(1)]$ does not contained in $[P_L,\ P_R]$, we get
$f(1)>P_R$. Depend on the position of $f(1)$, we consider two cases:
$f(1) \in (P_R,\ c^*)$ and $f(1) \ge c^*$.

If $f(1) \in (P_R,\ c^*)$, then there exists positive integer $k$ so
that $f(1) \in [c_k',\ c_{k-1}']$, by Lemma \ref{inequality},
$M(c_-)\ge ab a^kb^{k+1}$. We obtain
$$
(M_p-1)(M_q-1) \ge (aba^kb^{k+1}-1)(ab-1)>a^kb^{k+1}(ab-1)>1.
$$

If $f(1) \ge c^*$, then $f^2(1) \ge c$, which implies that
$M(f^2(1))\ge 1$ because the product $M(f^2(1))$ begin with $b$ and
it admits no consecutive $a$. We obtain $M_p \ge M(c_-) \ge
abbM(f^2(1)) \ge ab^2$. By Lemma \ref{baj}, $ab^2>1+b$. As a result,
$$
(M_p-1)(M_q-1) \ge (ab^2-1)(ab-1)>b(ab-1)>1.
$$
\vspace{0.2cm}

(ii) {\bf Subcase $c_*<f(0)<P_L$.}\vspace{0.2cm}

In this case, there exist $k \ge 1$ so that $f(0) \in (c_{k-1},\
c_k]$. Since $f(x)>f(0)$ for each $x \in A$, by Lemma \ref{M(x)}, we
know that $M(x_j) \ge 1$ for $j=1, 2, \ldots, s-1$, and $M(y_j) \ge
1$ for $j=1, 2, \ldots, t-1$. We have
$$M_q=M(c_+)M(x_1)\cdots
M(x_{s-1})W(x_s) \ge M(c_+)M(x_1)=aba^{k+1}b^k.$$
$$M_p\ge M(c_-)=ab M(f(1)) \ge ab.$$
Using Lemma \ref{inequality}, we conclude that
$$
(M_p-1)(M_q-1) \ge (aba^{k+1}b^k-1)(ab-1)>a^{k+1}b^k(ab-1)>1.
$$

\vspace{0.2cm} (iii) {\bf Subcase $f(0) \le c_*$.}\vspace{0.2cm}

Let $i$ be the minimal positive integer so that $f^i(0)>c$. Each
orbit can stay on the left of $c$ at most $i$ consecutive times. At
first, we conclude that
\begin{equation}\label{bai}
M(x) \ge ba^i \ \ \ \ \ \ \ \ \ \ x >c_*.
\end{equation}
In fact, one can write $M(x)=aUba^m$, where $m \le i-1$, and $U\ge
1$ because $U$ begin with $b$ and it admits no consecutive $a$. So
we have $M(x) \ge aba^{i-1}=ba^i$ because and $a<1$.

In what follows, we shall prove
\begin{equation}\label{length-3}
M_q \ge ba^{i-1},\ \ \ \ \ \ \ M_p \ge ba.
\end{equation}

\vspace{0.2cm}

{\bf Claim 1:} $M_q \ge ba^{i-1}$. \vspace{0.2cm}

Claim 1 will be proved in two separated cases: $ba^i \ge 1$ and
$ba^{i}<1 $.

Suppose $ba^i \ge 1$. By  Lemma \ref{decomposition}, Lemma
\ref{M(x)} and (\ref{bai}),
$$M_q=\prod_{m=0}^{r-1}f'(f^m(c_+))=M(c_+)M(x_1)\cdots M(x_{s-1})W(x_s)\ge M(c_+)=ba^{i-1}.$$

Now we suppose $ba^i<1$. By Lemma \ref{baj}, we know that
$ba^{i-1}>1+a+\cdots+a^{i-2}>1$. By Lemma \ref{decomposition},
$$
M_q =\prod_{m=0}^{r-1}f'(f^m(c_+))=M(c_+)M(x_1)\cdots M(x_{s-1})
W(x_s),
$$
where $E=f'(x_s)f'(f(x_s))\cdots f'(f^{r-1}(c_+)) \ge 1$.

In what follows we show that
$$M=M(x_1)M(x_2)\cdots M(x_{s-1}) \ge
1,$$ which implies our Claim $M_q \ge M(c_+)=ba^{i-1}$.

By Lemma \ref{M(x)-1}, $M(x_j) \ge 1$ for all $x_j>c_1$. So $M \ge
1$ if there is no $x_j$ is smaller than $c_1$.

Suppose there are some $j$ so that $M(x_j)<1$. We denote them as
$j_1<j_2<\cdots$. According to Lemma \ref{control}, $ba^i<1$ implies
$c_1<f^{i-1}(0)<c$. Using Lemma \ref{M(x)-1} we get $M(x_1)>1$. As a
result, we have $j_1>1$. By Lemma \ref{control}, we know that
$x_{j_1} \in (c_*,\ c_1]$, and $M(x_{j_1})=ba^i$ because each orbit
can stay on the left of $c$ at most $i$ consecutive times and
$ba^{i-1}>1$.

Let $k_1=max\{t: x_t>c_1, t<j_1\}$. It follows from Lemma
\ref{control} and Lemma \ref{M(x)-1} that $1\le k_1<j_1$ and
$x_{k_1}>c_1$, which, together with $ab>1$, implies $M(x_{k_1}) \ge
ababa^m$. Moreover, we conclude that $m<i-1$, because $m=i-1$
implies $x_{k_1}>c_1$ by Lemma \ref{control}. We obtain $M(x_{k_1})
\ge ababa^{i-2}$. Therefore, $M(x_{k_1})M(x_{j_1}) \ge
ababa^{i-2}ba^i=(ba^2)(ba^{i-1})(ba^{i-1}) \ge 1$.

By similar arguments, one can find $j_1<k_2<j_2$ so that
$M(x_{k_2})M(x_{j_2}) \ge 1$. Repeat the above procedures several
times if possible, we conclude that $M \ge 1$. Therefore, $M_q \ge
M(c_+)=ba^{i-1}$.

Claim 1 is true.

\vspace{0.2cm} {\bf Claim 2:} $M_p\ge ab.$\vspace{0.2cm}

Since the orbit
$$O_{\ell}(c_-)=\{c_-,\ 1,\ f(1),\ \ldots,\ f^{\ell-1}(c_-) \}$$ exits
$A$ exact $s (\ge 0)$ times, and the first point after $j$th exit is
$y_j$, we conclude that
\begin{equation} %\label{eq1}
\begin{split}
M_p=(f^{\ell})'(c_-)&=M(c_-)M(y_1)\cdots M(y_{t-1}) W \\
 &=abM(f(1))M(y_1)\cdots M(y_{t-1}) W
 \end{split}
 \end{equation}
where $W=f'(y_t)f'(f(y_t))\cdots f'(f^{\ell-1}(c_-))$.

Using the same arguments in the proof of Claim 1, one can show that
both $M(f(1))M(y_1)\cdots M(y_{t-1})$ and $W$ are greater than $1$.
Claim 2 holds. \vspace{0.3cm}

Using (\ref{length-3}) and Lemma \ref{baj},
$$
(M_p-1)(M_q-1) \ge (ab-1)(ba^{i-1}-1)>(ab-1)a^{i-2}>1.
$$
So the expansion of nice pair (\ref{expansion}) is proved when
$a<1<b$.

 \vspace{0.2cm}
{\bf Case C:} $a>1>b$.
 \vspace{0.2cm}

One can adapt the proof of the case $a<1<b$ to this case by using
the first exit decomposition of $M_p$ and $M_q$ with respect to the
set $B=[c^*,\ 1]$. $\hfill \Box$

 %\vspace{0.2cm}

%{\bf Theorem D.} {\it If $f$ is a piecewise linear Lorenz map with
%average slope greater than $1$, then each renormalization of the
%piecewise linear Lorenz map $f_{a,b,c}$ is periodic.} \vspace{0.1cm}

%\setcounter{equation}{0}
\section{Proof of Main Theorem' }\ \

Now we are ready to prove the Main Theorem', which, together with
Proposition 1, implies our Main Theorem.

\begin{proof}

It is proved in \cite{DFY} that a piecewise linear Lorenz map that
expand on average is always expanding. During the proof, we denote
the piecewise linear Lorenz map $f_{a,b,c}$ by $f \in \mathcal {L}$.

\vspace{0.2cm} {\bf Step 1.} Since the renormalization of piecewise
linear Lorenz map is still piecewise linear, in order to prove each
renormalization of $f$ is periodic, it is necessary to show that the
minimal renormalization of any renormalizable piecewise linear
Lorenz map is always periodic.

If $f$ does not satisfy the additional condition $1-ac<c<b(1-c)$, by
Proposition \ref{trivial renormalization}, there is an expanding
Lorenz map $g$ with minimal period $\kappa(g)<\kappa(f)$, such that
$f$ is renormalizable if and only if $g$ is renormalizable, and if
$f$ is renormalizable, then minimal renormalization of $f$ is
periodic if and only if the minimal renormalization of $g$ is
periodic. Furthermore, since $f$ is piecewise linear with
$ac+b(1-c)>1$, $g \in \mathcal {L}$.

Applying Proposition  \ref{trivial renormalization} several times if
necessary, we can assume that $\kappa(f) \le 2$ (see Corollary
\ref{cor}). It follows from Proposition  \ref{trivial
renormalization} that $f \in \mathcal {L}$ can not be renormalized
trivially if and only if either $\kappa(f)=1$ or $f \in \mathcal
{L}_2$. Since any expanding Lorenz map with $\kappa(f)=1$ is prime,
we only need to consider the case $f \in \mathcal {L}_2$.

\vspace{0.2cm}{\bf Step 2.} Suppose that $f \in \mathcal {L}_2$. Let
$O=\{P_L, \ P_R\}$  be the 2-periodic points of $f$, and
$P_L<c<P_R$, $D=\bigcup_{n\ge 0}f^{-n}(O)$ be the minimal completely
invariant closed set of $f$. We shall prove $f$ is prime if $D \neq
O$ by contradiction.

Now suppose $f$ is not prime, according to Theorem A in \cite{D},
the minimal renormalization map of $f$ is $R f$,
\begin{eqnarray*} %\label{renormal-linear}
R f(x)=\left \{ \begin{array}{ll}
f^{\ell}(x) & x \in [f^{r}(c_+),  c) \\
f^{r}(x) & x \in (c,  f^{\ell}(c_-) ],
\end{array}
\right.
\end{eqnarray*}
where $p=\sup \{x<c: x \in D\},\   q=\inf \{x>c: x \in D\}$, and
$\ell$ and $r$ are the maximal integers so that $f^{\ell}$ and
 $f^r$ is continuous on $(p,\ c)$ and $(c,\ q)$, respectively.
Obviously, $\{p,\ q\}$ is a nice pair.

Put $L=(p,\ c)$, $R=(c,\ q)$, $M_p=(f^{\ell})'(p)$ and
$M_q=(f^r)'(q)$. Since $Rf$ is a piecewise linear Lorenz map, we
have
\begin{eqnarray*}
|f^{\ell}(L)|=|f^{\ell}((p,\ c))|&=&|(p,\ f^{\ell}(c_-))|=M_p|L|\leq |L|\ +\ |R| \\
 |f^{r}(R)|= |f^{r}((c,\ q))|&=&|(f^{r}(c_+),\ q)|=M_q|R|\leq |L|\ +\
 |R|,
\end{eqnarray*}
which implies
\begin{equation}\label{ineqn0}
(M_p-1)(M_q-1) \le 1.
\end{equation}

On the other hand, if $D \neq O$, then $[f(0),\ f(1)] \nsubseteq
[P_L,\ P_R]$ by (\ref{periodic renormalization}). According to
Theorem \ref{lower bound}, we have
$$(M_p-1)(M_q-1)>1$$
because $\{p,\ q\}$ is a nice pair. We obtain a contradiction.

It follows that $f$ is prime if $D \neq O$. So we conclude that the
minimal renormalization of $f$ is periodic. As a result, each
renormalization of $f$ is periodic.

\vspace{0.2cm}{\bf Step 3.} Now we show that $f$ can only be
renormalized finite times. If $f$ is renormalizable, then the
minimal renormalization $Rf$ is a $\beta$-transformation because
$Rf$ is a periodic renormalization indicates $M_p=M_q$. So $g:=Rf$
is a $\beta$-transformation with slope $M_p$, which can be
renormalized at most finite times by (\ref{ineqn0}).   As a result,
$f$ can be renormalized at most finite times.

\end{proof}

\vspace{0.3cm}

{\bf Acknowledgements:}  This work is partially supported by a grant
from the Spanish Ministry (No.SB2004-0149) and grants from NSFC
(Nos.60534080,70571079) in China. Ding thanks Centre de Recerca
Matem$\grave{a}$tica for the hospitality and facilities.

\end{document}